\documentclass{article} 

\usepackage{amsmath, amssymb, amsthm}
\usepackage{color}
\usepackage[dvipdfm,a4paper,left=30mm,right=30mm,top=35mm,bottom=35mm]{geometry}
\usepackage[titletoc, toc, title]{appendix}
\usepackage{bbm}

\begin{document}
\renewenvironment{proof}{ \noindent {\bfseries Proof.}}{qed}
\newtheorem{defin}{Definition}
\newtheorem{theorem}{Theorem}
\newtheorem{remark}{Remark}
\newtheorem{proposition}{Proposition}
\newtheorem{lemma}{Lemma}
\newtheorem{cor}{Corollary}
\def\begproof{\noindent{\bf Proof: }}
\def\endproof{\quad\vrule height4pt width4pt depth0pt \medskip}
\def\div{\nabla\cdot}
\def\rot{\nabla\times}
\def\sign{{\rm sign}}
\def\arsinh{{\rm arsinh}}
\def\arcosh{{\rm arcosh}}
\def\diag{{\rm diag}}
\def\const{{\rm const}}
\def\eps{\varepsilon}
\def\phi{\varphi}
\def\theta{\vartheta}
\newcommand{\Bchi}{\mbox{$\hspace{0.12em}\shortmid\hspace{-0.62em}\chi$}}
\def\C{\hbox{\rlap{\kern.24em\raise.1ex\hbox
      {\vrule height1.3ex width.9pt}}C}}
\def\R{\mathbb{R}}
\def\P{\hbox{\rlap{I}\kern.16em P}}
\def\Q{\hbox{\rlap{\kern.24em\raise.1ex\hbox
      {\vrule height1.3ex width.9pt}}Q}}
\def\M{\hbox{\rlap{I}\kern.16em\rlap{I}M}}
\def\N{\hbox{\rlap{I}\kern.16em\rlap{I}N}}
\def\Z{\hbox{\rlap{Z}\kern.20em Z}}
\def\K{\mathcal{K}}
\def\({\begin{eqnarray}}
\def\){\end{eqnarray}}
\def\[{\begin{eqnarray*}}
\def\]{\end{eqnarray*}}
\def\part#1#2{{\partial #1\over\partial #2}}
\def\partk#1#2#3{{\partial^#3 #1\over\partial #2^#3}}
\def\mat#1{{D #1\over Dt}}
\def\dt{\, \partial_t}
\def\as{_*}
\def\d{\, \text{d}}
\def\e{ \text{e}}
\def\D{\mathcal{D}}
\def\pmb#1{\setbox0=\hbox{$#1$}
  \kern-.025em\copy0\kern-\wd0
  \kern-.05em\copy0\kern-\wd0
  \kern-.025em\raise.0433em\box0 }
\def\bar{\overline}
\def\lbar{\underline}
\def\fref#1{(\ref{#1})}

\begin{center}
{\LARGE Fractional diffusion limit of a linear kinetic equation in a bounded domain}
\bigskip

{\large P. Aceves-S\'anchez}\footnote{Fakult\"at f\"ur Mathematik, Universit\"at Wien.} and
{\large C. Schmeiser}$^1$
\end{center}
\vskip 1cm

\noindent{\bf Abstract.} A version of fractional diffusion on bounded domains, subject to 'homogeneous Dirichlet
boundary conditions' is derived from a kinetic transport model with homogeneous inflow boundary conditions.
For nonconvex domains, the result differs from standard formulations. It can be interpreted as the forward
Kolmogorow equation of a stochastic process with jumps along straight lines, remaining inside the domain.
\vskip 1cm

\medskip

\vskip 1cm

\noindent{\bf Acknowledgment:} This work has been supported by the PhD program {\em Dissipation and
Dispersion in Nonlinear PDEs}, supported by the Austrian Science Fund,
grant no. W1245, and by the Vienna Science and Technology Fund,
grant no. LS13-029. P.A.S also acknowledges support from the Consejo
Nacional de Ciencia y Tecnologia of Mexico.




\section{Introduction}

This work is an extension to bounded domains of earlier efforts \cite{ben2011anomalous,Mellet2010,MR2763032} to derive fractional diffusion equations
from kinetic transport models. This raises the issue of the inclusion of boundary effects, which can, however,
not be reduced to boundary conditions since fractional diffusion is a nonlocal process. Our main result is the
derivation of a new way of realizing 'homogeneous Dirichlet boundary conditions', coinciding on convex domains 
with an already established model, see e.g. \cite{Fukushima+2010}.

Let $\Omega\subset\R^d$ denote a bounded domain with smooth boundary. We shall study the asymptotic 
behavior as $\varepsilon > 0$ tends to zero of the kinetic relaxation model
\begin{align}
\varepsilon^{ \alpha} \partial_t f_\varepsilon + \varepsilon v \cdot \nabla_x f_\varepsilon = Q ( f_\eps) :=
\int_{ \R^d} M f'_\varepsilon - M' f_\varepsilon \d v'  \,,\label{maineq}
\end{align} 
with $f_\varepsilon = f_\varepsilon(x,v,t)$, $(x,v,t)\in \Omega\times\R^d\times [0,\infty)$
(where the superscript $^\prime$ denotes evaluation at $v'$), subject to zero inflow boundary conditions and well prepared initial data:
\begin{align}
f _\varepsilon( x, v, t) & = 0 \hspace{4cm} \text{ for } (x,v)\in\Gamma^-\,, \quad t>0 \,,\label{mainboundary}\\
f_\varepsilon ( x, v, 0) & = f^{in} ( x, v) := \rho^{in} ( x) M ( v) \qquad \text{for } (x,v)\in \Omega \,, \label{maininitial}
\end{align} 
with $\Gamma^\pm = \{ ( x, v) | \, \, x \in \partial \Omega, \, \, \mbox{sign}(v \cdot \nu ( x)) = \pm 1 \}$,
where $\nu$ denotes the unit outward normal along $\partial\Omega$. We assume a 'fat-tailed' equilibrium
distribution $M$, satisfying
\begin{eqnarray}
 &&M ( v) = 1/| v|^{ d + \alpha} \quad \text{ for } |v| \geq 1 \,, \mbox{ with } 0 < \alpha < 2 \,, \label{eq:eqdecay} \\
\label{eq:unbiased}
 &&M( v) > 0, \qquad M( v) = M( -v) \qquad \mbox{ for all } v \in \R^d \,, \\
\label{eq:mass}
&& M \in L^\infty ( \R^d) \,,\qquad \mbox{ and } \qquad \int_{ \R^d} M( v) \d v = 1 \,.
\end{eqnarray}
Note that these assumptions imply that $M$ does not have finite second order moments.

The translation of homogeneous Dirichlet boundary conditions to fractional diffusion induce a certain behaviour of
solutions close to the boundary. The domain of the fractional diffusion operator, we shall derive, contains test 
functions in
\begin{equation}\label{def:testfunc}
  \mathcal{D}_\Omega := \{\varphi\in C_0^\infty(\overline\Omega \times [0,\infty)):\, \delta(x)^{-2}\varphi(x,t) \mbox{ bounded}\} \,,
\end{equation}
where $\delta ( x) := \mbox{dist} ( x, \partial \Omega)$ denotes
the distance of a point $x \in \Omega$ to the boundary.

A convenient functional analytic setting for the main result of this paper is the $L^2$-space 
$L^2_{ M^{ -1}} ( \Omega \times \R^d)$ of functions of $(x,v)$ with weight $1/M(v)$.

\begin{theorem}\label{maintheo}
 Let $\rho^{in}\in L^2(\Omega)$, and let $f_\varepsilon$ be the solution of \eqref{maineq}--\eqref{maininitial}.
 Then, for any $T > 0$, there exists $\rho \in L^\infty ( 0, T ; L^2 ( \Omega))$ such that
 $f_\varepsilon(x,v,t)\to \rho(x,t)M(v)$ as $\varepsilon \rightarrow 0$, in $L^\infty ( 0, T; L^2_{ M^{-1}} ( \Omega \times \R^{ d} ))$ weak-$\star$, and
 $\rho$ satisfies 
 \begin{align}
  \int_\Omega \rho^{in} \varphi(t=0)dx + \int_0^\infty \int_\Omega \rho\,\partial_t \varphi \,dx\,dt  = 
   \int_0^\infty \int_\Omega \rho( h_\alpha \varphi - \mathcal{L}_\alpha ( \varphi))dx\,dt ,  
    \label{mainfrac} 
 \end{align}
for all $\varphi\in\mathcal{D}_\Omega$, with
 \[
  \mathcal{L}_\alpha ( \varphi) ( x, t) = \Gamma(\alpha+1)\emph{P.V.} \int_{ \{ w \in \R^d : [ x, x + w ] \subset \Omega \}} \frac{ \varphi ( x + w, t) - \varphi ( x, t)}{ | w|^{ d + \alpha}} \d w,
 \]
and
 \begin{equation}\label{eq:hfunc}
  h_\alpha ( x) = \int_{ \R^d} \frac{ 1}{ | w|^{ d + \alpha}} \e^{ - \frac{ | x - x_0 ( x, w)|}{ | w|}} \d w \,,
 \end{equation}
where $[x,y]$, $x,y\in\R^d$, denotes the straight line segment connecting $x$ and $y$, and $x_0(x,w)$ is the point
closest to $x$ in the intersection of $\partial \Omega$ with the ray starting at $x$ in the direction $w$.
\end{theorem}

The function $h_\alpha$ is well defined by \eqref{eq:hfunc} and converges to $\infty$ when $x \to \partial \Omega$, see Proposition \ref{prop:hbound} in Section \ref{sec:disc}.

\begin{remark}
Theorem \ref{maintheo} remains true with slightly modified proofs for generalized versions of the model. For
example, \eqref{eq:eqdecay} may be replaced by the more general condition
\begin{equation}\label{eq:geqdecay}
M( v) \sim 1 / | v|^{ d + \alpha} \qquad \mbox{ as } | v| \to \infty.
\end{equation}
An example coming from stochastic analysis is the probability density function of an $\alpha$-stable process, see \cite{Bogdan+2007}.

\end{remark}

\begin{remark}
Another possible generalization is to permit a more general collision operator, satisfying the micro-reversibility principle: \[
Q( f) = \int_{ \R^d} [\sigma ( v , v') M( v) f ( v') - \sigma ( v' , v) M( v') f ( v)] \d v'
\]
where the cross-section $\sigma$ is symmetric, i.e. $\sigma ( v, v') = \sigma ( v', v)$,  $v, v'$ in $\R^d$, and bounded
from above and away from zero:
\[
 0 < \nu_1 \leq \sigma ( v, v') \leq \nu_2 < \infty \,.
\]
\end{remark}

The derivation of macroscopic limits from kinetic equations when the collision kernel has a Maxwellian as an equilibrium distribution is a classical
problem studied in the pioneering works \cite{Wigner1961}, \cite{Habetler+1975}, and \cite{Larsen+1974}. 
Here the essential properties of the equilibrium distribution are vanishing mean velocity and finite second order
moments. In the case where the 
equilibrium distribution is heavy-tailed, the problem was first studied for relaxation type collision operators in \cite{MR2763032}, \cite{Mellet2010} and \cite{ben2011anomalous}, from an analytical point
of view and in \cite{Jara+2009} with a probabilistic approach, obtaining as a macroscopic limit a fractional heat equation. These are results on whole space, and they have recently been extended to collision operators of
fractional Fokker-Planck type \cite{Cesbron+2012} and to the derivation of fractional diffusion with drift \cite{AcevesSanchez+2016,AcevesSancheza+2016,AcevesSanchezc+2016}. 
The proofs of most of these results are based on the moment method introduced in \cite{Mellet2010},
which will also be used here.

To find an appropriate definition of fractional diffusion in a bounded domain is not obvious since it describes the
probability distribution of a jump process. The formulation of appropriate models as macroscopic limits of
kinetic equations is the subject of this work and of the very recent contribution \cite{Ludovic2016}, where the problem
of deriving a fractional heat equation from a kinetic fractional-Fokker-Planck equation is tackled with zero inflow and specular reflection boundary conditions, where
the spatial domain is a circle. The main differences between this work and \cite{Ludovic2016} are that we use
a relaxation type collision operator, we only consider inflow boundary conditions, but we permit general,
in particular nonconvex, position domains. 

There are several equivalent definitions of the fractional Laplacian in the whole domain (see \cite{Kwa2015}), however,  for bounded domains there are
different definitions, depending on the details of the underlying stochastic process. For instance, if we consider the stochastic process consisting of a fractional
Brownian motion with an $\alpha / 2-$stable subordinator and killed upon leaving the domain it has as infinitesimal generator the restricted fractional
Laplacian (see \cite{Fukushima+2010})
\begin{equation}\label{def:fracLap}
  -( - \Delta |_\Omega)^{ \alpha / 2} \varphi ( x) := c_{ d, \alpha} \mbox{ P.V.} \int_{ \R^d} 
  \frac{ \varphi(y){\rm\bf 1}_\Omega(y) - \varphi(x)}{ | x - y|^{ d + \alpha}} \d y \,,\qquad c_{d,\alpha}>0\,.
\end{equation}
This operator has also been derived in \cite{Ludovic2016} as macroscopic limit of a kinetic equation in a circle, 
subject to zero inflow boundary conditions. The macroscopic operator of Theorem \ref{maintheo} can be written 
in the similar form,
\begin{equation}\label{def:our-fracLap}
  -h_\alpha \varphi + \mathcal{L}_\alpha(\varphi) = \Gamma(\alpha+1)\,\mbox{P.V.} 
    \int_{ \R^d} \frac{ \varphi(y){\rm\bf 1}_{\mathcal{S}_\Omega(x)}(y) - \varphi(x)}{ | x - y|^{ d + \alpha}} dy \,,
\end{equation}
where $\mathcal{S}_\Omega(x)$ denotes the biggest star-shaped subdomain of $\Omega$ with center in $x$.
Obviously, \eqref{def:fracLap} and \eqref{def:our-fracLap} coincide for convex $\Omega$ (the situation of 
\cite{Ludovic2016}). The difference in the stochastic process interpretations of \eqref{def:fracLap} and 
\eqref{def:our-fracLap} is that in the latter jumps are only permitted along straight lines, which do not leave the
domain.

For completeness we also mention the spectral fractional Laplacian defined as follows: The operator $-\Delta$ 
subject to homogeneous Dirichlet boundary conditions along $\partial\Omega$ has positive eigenvalues $ 0 < \lambda_1 \leq \lambda_2 \ldots$ with corresponding normalized eigenfunctions $\{ e_k \}_{ k \geq 1}$. 
The spectral fractional Laplacian (subject to homogeneous Dirichlet boundary conditions) is defined by
\begin{equation}\label{def:specfrac}
 ( - \Delta_\Omega)^{\alpha / 2} \varphi ( x) := \sum_{ i=1}^{ \infty} \lambda_i^{ \alpha / 2} e_i(x)
  \int_{\Omega} e_i(y) \varphi(y)dy\,.
\end{equation}
It can also be interpreted as generating a stochastic process (see \cite{Chen+1998}). A representation formula
similar to \eqref{def:fracLap} and \eqref{def:our-fracLap} has been derived in \cite{Song+2003}:
\[
( - \Delta_\Omega)^{ \alpha / 2} \varphi ( x) = c_{ d, \alpha} \mbox{ P.V.} \int_\Omega [ \varphi( x) - \varphi ( y)] J ( x, y) \d y + c_{ d, \alpha} \, \kappa ( x) \varphi ( x), \quad \mbox{ for } x \in \Omega 
\]
where the functions $J$ and $\kappa$ and the constant $c_{ d, \alpha}$ satisfy (with positive  constants $C_1$, $C_2$ and $C_3$)
\[
C_1 \delta ( x) \delta ( y) \leq J ( x, y)  \leq C_2 \min \bigg( \frac{ 1}{ | x - y|^{ d + \alpha}}, \frac{ \delta ( x) \delta ( y)}{ | x - y|^{ d + 2 + \alpha}} \bigg), 
\]
and
\[
C_3^{ -1} \delta^{ -\alpha} ( x)  \leq \kappa ( x)  \leq C_3 \delta ^{ - \alpha} ( x). 
\]
In \cite{Servadei+2014} it is proven that the two operators $( - \Delta_\Omega)^{\alpha / 2}$ and $( - \Delta |_\Omega)^{ \alpha / 2}$ are different since,
for instance, the eigenfunctions of the former are smooth up to the boundary whereas the eigenfunctions of the latter are no better than
H\"older continuous up to the boundary. In recent years fractional Laplace operators have been extensively used since they seem to be more suitable 
for the description of phenomena such as contaminants propagating in water \cite{Benson+2001}, plasma physics \cite{Del+2005}, among many  others (see
\cite{Vazquez2014857} and \cite{DiNezza2012521}). However, there is some literature where for the fractional Laplacian on bounded domains the definitions \eqref{def:fracLap} and \eqref{def:specfrac} are used interchangeably, thus leading to false results.

\section{Uniform estimates and modified test functions}

It is a standard result of kinetic theory that the initial-boundary value problem \eqref{maineq}--\eqref{maininitial} with an
equilibrium distribution $M$ satisfying \eqref{eq:eqdecay}--\eqref{eq:mass} and an initial position density 
$\rho^{in}\in L^1(dx)$ has a unique solution, which is nonnegative, if the same holds for $\rho^{in}$ (see, e.g. \cite{MR1295030}, Chapter XXI). This will be assumed
in the following, where we always denote by $dx$, $dv$, and $dt$ the Lebesgue measures on $\Omega$, $\R^d$, and, respectively,
$(0,\infty)$. We start with standard estimates:

\begin{lemma}\label{lem:uniform}
Let $\rho^{in} \in L^2_+(dx)$. Then the solution $f_\eps$ of \eqref{maineq}--\eqref{maininitial}
satisfies 
$$
  f_\eps\in L^\infty( dt,\,L^2_+(dx\,dv/M)) \,\qquad \mbox{uniformly as } \eps\to 0 \,,
$$ 
and, with $\rho_\eps := \rho_{f_\eps}$,
$$
   f_\eps - \rho_\eps M = O(\eps^{\alpha/2}) \qquad\mbox{in } L^2(dx\,dv\,dt/M) \,,\qquad\mbox{as } \eps\to 0 \,.
$$
\end{lemma}

\begin{proof}
Multiplication of \eqref{maineq} by $f_\eps/M$, integration with respect to $x$ and $v$, the divergence 
theorem, and the boundary condition \eqref{mainboundary} yield
\begin{eqnarray}
 \frac{ \varepsilon^\alpha}{2} \frac{d}{d t} \int_\Omega \int_{ \R^d} \frac{ f^2_\eps}{M} dv\,dx + 
 \eps \int_{\Gamma^+} v \cdot \nu  \frac{ f_\eps^2}{2 M} dv\,dx 
  &=& \int_{ \Omega} \int_{ \R^d} Q(f_\eps) \frac{ f_\eps}{M} dv\,dx \nonumber\\
  &=&  - \| f _\eps- \rho_\eps M \|^2_{ L^2(dx\,dv/M)}  \,,\label{eq:weakform}
\end{eqnarray}
where the second equality is a well known fact and the result of a straightforward computation
(see, e.g. \cite{MR1803225}). The nonnegativity of the second term and an integration with respect to $t$ over $(0,T)$ give
\[
 \frac{ \eps^\alpha}{2} \|f_\eps(\cdot,\cdot,T)\|_{L^2(dx\,dv/M)}^2 + \int_0^T \|f_\eps - \rho_\eps M\|_{L^2(dx\,dv/M)}^2 dt 
   \leq \frac{\eps^\alpha}{2} \|\rho^{in}\|_{L^2(dx)}^2 \,,
\]
completing the proof.
\end{proof}

For the proof of Theorem \ref{maintheo} we employ the moment method introduced in \cite{Mellet2010}, which relies on
test functions solving a suitably chosen adjoint problem. For given $\varphi\in\mathcal{D}_\Omega$ 
the function $\chi_\eps(x,v,t)$ is the solution of the stationary kinetic equation
\begin{equation}\label{auxiliaryeq}
 \chi_\eps - \eps v \cdot \nabla_x \chi_\eps = \varphi \,,
\end{equation}
subject to the inflow boundary condition
\begin{equation} \label{auxiliarybc}
  \chi_\eps= 0 \qquad\mbox{on } \Gamma^+ \,.
\end{equation}
Note that the left hand side of \eqref{auxiliaryeq} is an adjoint version of a part of \eqref{maineq}, where only the loss term of the
collision operator and the transport operator have been kept.

We can readily solve \eqref{auxiliaryeq}, \eqref{auxiliarybc} via the method of characteristics, obtaining
\begin{equation}\label{auxiliarysol}
\chi_\eps(x,v,t) = \int_0^{ r ( x, v) / \eps} \e^{-s} \varphi ( x + \eps s v , t) ds \,, \qquad
\mbox{where}\quad r ( x, v ) = \frac{| x - x_0 ( x, v)|}{|v|} \,,
\end{equation} 
and $x_0 ( x, v)$ is the point closest to $x$ in the intersection of $\partial \Omega$ 
and the ray starting at $x$ with direction $v$. In the following a different representation will be convenient:
\begin{equation}\label{auxiliarysol1}
 \chi_\eps(x,v,t) = \varphi(x,t) \left( 1 - \e^{ -r ( x, v)/\eps} \right) 
 + \int_0^{ r( x, v)/\eps} \e^{ -s} [ \varphi ( x + \eps sv, t) - \varphi ( x, t)] ds \,.
\end{equation}
This already shows the main difference to the whole space situation \cite{Mellet2010}, which is the boundary layer
correction in the parenthesis on the right hand side of \eqref{auxiliarysol1}. 

In the following we shall need a uniform boundedness result.

\begin{lemma} \label{lem:phi}
 Let $\varphi\in\mathcal{D}_\Omega$ and let $\chi_\eps$ be given by \eqref{auxiliarysol}. Then
 $$
   \|\chi_\eps\|_{L^2(M\,dx\,dv\,dt)} \le \|\varphi\|_{L^2(dx\,dt)} \,,\qquad
   \|\partial_t\chi_\eps\|_{L^2(M\,dx\,dv\,dt)} \le \|\partial_t\varphi\|_{L^2(dx\,dt)} \,.
 $$
 \end{lemma}

\begin{proof}
Multiplication of \eqref{auxiliaryeq} by $M\chi_\eps$ and integration with respect to $v$ gives
$$
  \|\chi_\eps\|_{L^2(M\,dv)}^2 - \frac{\eps}{2}\nabla_x\cdot \int_{\R^d} vM\chi_\eps^2 dv = \varphi\int_{\R^d} M\chi_\eps dv
  \le |\varphi| \,\|\chi_\eps\|_{L^2(M\,dv)} \,,
$$
where the Cauchy-Schwarz inequality and the normalization of $M$ has been used. Integration with respect to $x$ and $t$,
the divergence theorem, and the boundary condition \eqref{auxiliarybc} for $\chi_\eps$ lead to
$$
  \|\chi_\eps\|_{L^2(M\,dx\,dv\,dt)}^2 - \frac{\eps}{2}\int_0^\infty \int_{\Gamma^-} \nu\cdot vM\chi_\eps^2 dv\,d\sigma\,dt 
  \le \|\varphi\|_{L^2(dx\,dt)} \|\chi_\eps\|_{L^2(M\,dx\,dv\,dt)} \,,
$$
completing the proof of the first inequality. The proof of the second is analogous after differentiation of \eqref{auxiliaryeq}
 with respect to $t$.
\end{proof}
 
\section{Proof of Theorem \ref{maintheo}}

With $\varphi \in \mathcal{D}_\Omega$ and $\chi_\eps$ defined by \eqref{auxiliarysol},
multiplication of \eqref{maineq} by $\chi_\eps$ and integration with respect to $x$, $v$ and $t$ gives
\begin{align}
& - \int_0^\infty \int_{\R^d} \int_\Omega f_\eps \partial_t \chi_\eps dx\,dv\,dt 
    -  \int_{\R^d} \int_\Omega \rho^{ in}M \chi_\eps(t=0) dx\,dv \nonumber \\
& \qquad \qquad \qquad \qquad \qquad = \eps^{-\alpha} \int_0^\infty \int_{ \R^d} \int_\Omega (\rho_\eps M \chi_\eps 
   - f_\eps \chi_\eps + f_\eps \eps v \cdot \nabla_x \chi_\eps) dx\,dv\,dt \nonumber \\
& \qquad \qquad \qquad \qquad \qquad = \int_0^\infty \int_\Omega \rho_\eps \left( \eps^{-\alpha} \int_{ \R^d} M    
   (\chi_\eps - \varphi) dv \right) dx\,dt \,. \label{weakform}
\end{align}

 \noindent In the sequel we shall need the following notation: For $x,y \in \R^d$ we denote by $[x,y]$ the line segment 
 connecting $x$ and $y$. Furthermore, we denote by $\mathcal{S}_\Omega(x)$ the largest star shaped subdomain of $\Omega$
 with center $x$, i.e. 
 $$
   \mathcal{S}_\Omega(x) := \{y\in\Omega:\, [x,y]\subset\Omega\}
$$
The heart of our analysis is the asymptotics for the term in parantheses on the right hand side of \eqref{weakform}.
 
\begin{lemma}\label{mainlemma}
 Let $\varphi\in\mathcal{D}_\Omega$ and let $\chi_\varepsilon$ be given by \eqref{auxiliarysol}. Then
 \begin{equation}\label{bterm}
  \lim_{\eps\to 0} \varepsilon^{ - \alpha} \int_{ \R^d} M (\chi_\eps - \varphi) dv
  =  -h_\alpha \varphi +  \mathcal{ L}_\alpha ( \varphi)
 \end{equation}
locally uniformly in $x$ and $t$, where 
 \[
  h_\alpha ( x) = \int_{ \R^d} \frac{ 1}{ |v|^{ d + \alpha}} \e^{ - \frac{ | x - x_0 ( x, v)|}{ |v|}} dv \,,
 \]
 \[
  \mathcal{ L}_\alpha ( \varphi)(x,t) = \Gamma(\alpha+1)\,\emph{P.V.} \int_{\mathcal{S}_\Omega(x)} 
    \frac{ \varphi (y,t) - \varphi (x,t)}{ |y-x|^{ d + \alpha}} dy \,.
 \]
\end{lemma}

\begin{proof} 
The representation \eqref{auxiliarysol1} of $\chi_\eps$ induces the splitting
$$
  \eps^{-\alpha} \int_{ \R^d} M ( \chi_\eps - \varphi)d v  = -h_\alpha^\eps \varphi + \mathcal{ L}_\alpha^\eps ( \varphi) \,,
$$
with
\begin{eqnarray*}
h_\alpha^\eps(x) &=& -\eps^{-\alpha}\int_{ \R^d} M(v) e^{-r(x,v)/\eps}d v \,,\\
\mathcal{ L}_\alpha^\eps ( \varphi)(x,t) &=& \eps^{-\alpha}\int_{ \R^d} \int_0^{ r( x, v)/\eps} M(v)e^{ -s} [ \varphi ( x + \eps sv, t) - \varphi ( x, t)] ds\,d v \,.
\end{eqnarray*}
We shall consider these parts separately. In both cases we shall start by proving that the small velocities do not contribute to the
limit. This splits the rest of the proof into 4 steps.
\medskip

\noindent{\it Step 1:} We consider the contribution to $h_\alpha^\eps$ coming from the small velocities. For $|v|\le 1$ we 
have $r(x,v) \ge \delta(x)$. Therefore
$$
  \eps^{-\alpha}\int_{ |v|\le 1} M(v) e^{-r(x,v)/\eps}d v \le \eps^{-\alpha} e^{-\delta(x)/\eps} 
  \le c\frac{\eps^{2-\alpha}}{\delta(x)^2}\,,
$$
since the map $z\mapsto z^2 e^{-z}$, $z\ge 0$, is bounded.
\medskip

\noindent{\it Step 2:} The previous step implies that $h_\alpha^\eps$ is asymptotically equivalent to
\begin{equation} \label{step2}
\eps^{-\alpha}\int_{ |v| > 1} |v|^{-d-\alpha} e^{-r(x,v)/\eps}d v \,.
\end{equation}
In this integral we make the coordinate transformation $w=\eps v$. Observing that 
$$
   \frac{r(x,w/\eps)}{\eps} = \frac{|x-x_0(x,w/\eps)|}{|w|} = r(x,w) \,,
$$
since $x_0(x,w/\eps) = x_0(x,w)$, the expression in \eqref{step2} is equal to
$$
  \int_{|w|> \eps} |w|^{-d-\alpha} e^{-r(x,w)} dw \,.
$$
For proving that this converges to $h_\alpha(x)$, we need to estimate
\begin{eqnarray*}
  \int_{|w|\le \eps} |w|^{-d-\alpha} e^{-r(x,w)} dw &\le& \int_{|w|\le \eps} |w|^{-d-\alpha} e^{-\delta(x)/|w|} dw
  = |S^d| \delta(x)^{-\alpha} \int_{\delta(x)/\eps}^\infty s^{\alpha-1}e^{-s}ds \\
  &\le&  |S^d| \frac{\eps^{2-\alpha}}{\delta(x)^2} 
  \sup_{\gamma\ge 0} \left(\gamma^{2-\alpha} \int_\gamma^\infty s^{\alpha-1}e^{-s}ds \right)\,.
\end{eqnarray*}
The supremum is finite since the integrand is bounded and decays exponentially as $s\to\infty$.

Combining this result with Step 1 shows that
$$
  |h_\alpha^\eps(x) - h_\alpha(x)| \le c\frac{\eps^{2-\alpha}}{\delta(x)^2} \,,
$$
implying pointwise convergence of $h_\alpha^\eps$ to $h_\alpha$ in $\Omega$. Since $|\varphi(x,t)| \le c\,\delta(x)^2$, 
the convergence of $h_\alpha^\eps\varphi$ to $h_\alpha\varphi$ is uniform in $(x,t)$.
\medskip

\noindent{\it Step 3:} We analyze the contributions from the small velocities to $\mathcal{ L}_\alpha^\eps ( \varphi)$. 
For the test function difference, we apply the Taylor expansion:
\begin{eqnarray*}
 && \left|\eps^{-\alpha}\int_{|v|\le 1} \int_0^{ r( x, v)/\eps} M(v)e^{ -s} \left( \eps s v\cdot\nabla_x\varphi (x, t) 
 + \frac{\eps^2 s^2}{2} v^{tr}\nabla_x^2\varphi (\hat x, t)v \right) ds\,d v \right| \\
 &&\le \left|\eps^{1-\alpha}\nabla_x\varphi (x, t)\cdot\int_{|v|\le 1} vM(v)\int_0^{ r( x, v)/\eps} se^{ -s} 
  ds\,d v \right| + \eps^{2-\alpha}c \int_{|v|\le 1} |v|^2 M(v) dv \int_0^\infty s^2 e^{-s}ds \,.
\end{eqnarray*}
In the first term on the right hand side we change the order of integration:
\begin{eqnarray*}
  &&\int_{|v|\le 1} vM(v)\int_0^{ r( x, v)/\eps} s e^{ -s} ds\,dv 
  = \int_0^\infty se^{ -s}  \int_{|v|\le 1,\, \eps s\le r(x,v)} vM(v) dv\,ds \\
  && = \int_0^{\delta(x)/\eps} s e^{ -s}  \int_{|v|\le 1} vM(v) dv\,ds +
  \int_{\delta(x)/\eps}^\infty se^{ -s}  \int_{|v|\le 1,\, \eps s\le r(x,v)} vM(v) dv\,ds
\end{eqnarray*}
In the first term on the right hand side, the restriction $\eps s\le r(x,v)$ can be omitted, since it is automatically satisfied
for $\eps s \le \delta(x) \le r(x,v)$. As a consequence this term vanishes by $M$ being even.
The last term can be estimated by 
$$
  \int_{\delta(x)/\eps}^\infty se^{-s}ds  \int_{|v|\le 1} |v|M(v) dv \le c\frac{\eps}{\delta(x)} 
  \sup_{\gamma\ge 0} \left( \gamma \int_\gamma^\infty s e^{-s}ds\right) \,.
$$
Since $\varphi\in\mathcal{D}_\Omega$ implies $|\nabla_x\phi(x,t)|\le c\delta(x)$, we have the result
$$
  \eps^{-\alpha}\int_{|v|\le 1} \int_0^{ r( x, v)/\eps} M(v)e^{ -s} [ \varphi ( x + \eps sv, t) - \varphi ( x, t)] ds\,dv 
  = O(\eps^{2-\alpha}) \,,
$$
uniformly in $(x,t)$.
\medskip

\noindent{\it Step 4:} It remains to consider
\begin{eqnarray}
  && \eps^{-\alpha}\int_{|v|>1} \int_0^{ r( x, v)/\eps} |v|^{-d-\alpha} e^{ -s} [ \varphi ( x + \eps sv, t) - \varphi ( x, t)] ds\,dv 
  \nonumber\\
  && = \int_{|w|>\eps} \int_0^{ r( x,w)} |w|^{-d-\alpha} e^{ -s} [ \varphi ( x + sw, t) - \varphi ( x, t)] ds\,dw \nonumber\\
  && = \int_0^\infty s^{d+\alpha}e^{-s}\int_{|w|>\eps,\,s<r(x,w)}  \frac{\varphi ( x + sw, t) - \varphi ( x, t)}{|sw|^{d+\alpha}} dw\,ds
  \label{step4}
\end{eqnarray}
By the coordinate transformation $x+sw=y$ the condition $s<r(x,w)$ becomes $s$-independent: 
$$
|x-y|<|x-x_0(x,y-x)| \quad\Longleftrightarrow\quad y\in \mathcal{S}_\Omega(x) \,.
$$
Therefore \eqref{step4} is equal to
$$
  \int_0^\infty s^\alpha e^{-s}\int_{\mathcal{S}_\Omega(x)\setminus B_{\eps s}(x)}  
  \frac{\varphi (y,t) - \varphi(x,t)}{|y-x|^{d+\alpha}} dy\,ds \,,
$$
where $B_r(x)$ denotes the ball with center $x$ and radius $r$. In order to prove that this converges to 
$\mathcal{ L}_\alpha ( \varphi)$, we need to show that
$$
  \int_0^\infty s^\alpha e^{-s}\int_{\mathcal{S}_\Omega(x)\cap B_{\eps s}(x)}  
  \frac{(y-x)\cdot\nabla_x\varphi(x,t) + (y-x)^{tr}\nabla_x^2\varphi(\hat x,t)(y-x)/2}{|y-x|^{d+\alpha}} dy\,ds
$$
tends to zero. The second term involving the Hessian of the test function can be estimated by
$$
  c \int_0^\infty s^\alpha e^{-s} \int_{B_{\eps s}(x)} |y-x|^{2-d-\alpha} dy\,ds = c\,\eps^{2-\alpha} \int_0^\infty s^2 e^{-s}ds \,.
$$
The estimation of the first term is more subtle. Actually, the integral with respect to $y$ has to be understood as a principal value
for $\alpha\ge 1$. Since
$$
  {\rm P.V.} \int_{B_r(x)} \frac{y-x}{|y-x|^{d+\alpha}} dy = 0 \,,\qquad \mbox{for } r>0 \,,
$$
and $B_{\eps s}(x) \subset \mathcal{S}_\Omega(x) $ for $\eps s < \delta(x)$, we have
\begin{eqnarray}
  && \int_0^\infty s^\alpha e^{-s} {\rm P.V.} \int_{\mathcal{S}_\Omega(x)\cap B_{\eps s}(x)}  
  \frac{(y-x)\cdot\nabla_x\varphi(x,t)}{|y-x|^{d+\alpha}} dy\,ds \label{step5}\\
  &&= \int_{\delta(x)/\eps}^\infty s^\alpha e^{-s} \int_{(\mathcal{S}_\Omega(x)\cap B_{\eps s}(x))\setminus B_{\delta(x)}}  
  \frac{(y-x)\cdot\nabla_x\varphi(x,t)}{|y-x|^{d+\alpha}} dy\,ds \,,\nonumber
\end{eqnarray}
which can be estimated by
\begin{eqnarray*}
   c\delta(x) \int_{\delta(x)/\eps}^\infty s^\alpha e^{-s} \int_{B_{\eps s}(x)\setminus B_{\delta(x)}}  
  |y-x|^{1-d-\alpha} dy\,ds 
   = c\delta(x) \int_{\delta(x)/\eps}^\infty s^\alpha e^{-s} \int_{\delta(x)}^{\eps s} r^{-\alpha} dr\, ds \,.
\end{eqnarray*}
With
$$
  \int_{\delta(x)}^{\eps s} r^{-\alpha} dr \le \left\{ \begin{array}{ll} c(\eps s)^{1-\alpha} \,, & \alpha < 1 \,,\\
               \log(\eps s/\delta(x)) \,, & \alpha=1\,,\\
               c\delta(x)^{1-\alpha} \,,& \alpha>1 \,,\end{array} \right.
$$
it is straightforward to obtain that \eqref{step5} is $O(\eps^{2-\alpha})$ for $\alpha\ne 1$ and $O(\eps\log(1/\eps))$ for 
$\alpha=1$, uniformly in $(x,t)$. This completes the proof of the uniform convergence of $\mathcal{L}_\alpha^\eps(\varphi)$
to $\mathcal{L}_\alpha(\varphi)$.
\end{proof}

 \begin{cor}\label{cor}
 Let $\varphi \in \mathcal{D}_\Omega$ and let $\chi_\eps$ be defined by \eqref{auxiliarysol}. Then
 \begin{equation}\label{keyestimate3}
 \lim_{\eps\to 0}\int_{ \R^d} M( v) [ \chi_\varepsilon ( x, v, t) - \varphi ( x, t)] d v = 
 \lim_{\eps\to 0}\int_{ \R^d} M( v) [ \partial_t \chi_\varepsilon ( x, v, t) - \partial_t \varphi ( x, t)] d v = 0, 
 \end{equation}
uniformly with respect to $(x,t)\in{\rm supp}(\varphi)$.
\end{cor}
 
 \begin{proof}
 The first statement is an immediate consequence of Lemma \ref{mainlemma}. The second statement follows, since 
 $\varphi \in \mathcal{D}_\Omega$ implies $\partial_t\varphi \in \mathcal{D}_\Omega$ and since the map 
 $\partial_t\varphi\mapsto\partial_t\chi_\eps$ is the same as  $\varphi\mapsto\chi_\eps$.
 \end{proof}
 
 The remaining steps in the proof of Theorem \ref{maintheo} are rather standard. As a consequence of Lemma \ref{lem:uniform}
 and of the estimate 
 $$
    |\rho_\eps| \le \|f_\eps\|_{L^2(dv/M)} \qquad\Longrightarrow\qquad \|\rho_\eps\|_{L^2(dx)} \le \|f_\eps\|_{L^2(dx\,dv/M)} \,,
 $$
 we obtain
 $$
 \rho_\eps\stackrel{*}{\rightharpoonup}\rho \quad\mbox{in } L^\infty(dt;\,L^2(dx)) \,,\qquad
 f_\eps\stackrel{*}{\rightharpoonup}\rho M \quad\mbox{in } L^\infty(dt;\,L^2(dx\,dv/M)) \,,
 $$
when restricting to subsequences. Now we are ready for passing to the limit in \eqref{weakform}. We decompose the first term by using
$$
  \int_{\R^d} f_\eps\partial_t\chi_\eps dv = \int_{\R^d} (f_\eps - \rho_\eps M)\partial_t\chi_\eps dv 
  + \rho_\eps\int_{\R^d} M\partial_t\chi_\eps dv \,.
$$
The first term on the right hand side tends to zero by
$$
  \left| \int_{\R^d} (f_\eps - \rho_\eps M)\partial_t\chi_\eps dv \right| 
    \le \|f_\eps - \rho_\eps M\|_{L^2(dv/M)} \|\partial_t\chi_\eps\|_{L^2(M\,dv)} \,,
$$
Lemma \ref{lem:uniform}, and Lemma \ref{lem:phi}. In the second term we may pass to the limit $\rho\,\partial_t\varphi$
by the weak$*$ convergence of $\rho_\eps$ and the strong convergence of $\int_{\R^d} M\partial_t\chi_\eps dv$
(Corollary \ref{cor}). The limit in the second term of \eqref{weakform} is a consequence of Corollary \ref{cor}. Finally, 
passing to the limit in the right hand side of \eqref{weakform} is justified by the weak$*$ convergence of $\rho_\eps$ and
by Lemma \ref{mainlemma}. This completes the proof of Theorem \ref{maintheo}.
 
\section{Discussion}\label{sec:disc} 
  
In this section we discuss properties of the fractional diffusion operator. First we show that the function $h_\alpha$
defined in \eqref{eq:hfunc} is well defined and tends to infinity at the boundary of $\Omega$.
\begin{proposition}\label{prop:hbound}
 Let $h_\alpha$ be defined by \eqref{eq:hfunc}, then there exists $C > 0$ such that
 \begin{equation}\label{hbeha}
  0< h_\alpha ( x) \leq C \delta(x)^{ - \alpha} \,,\qquad x\in\Omega \,.
 \end{equation}
 \end{proposition}

\begin{proof}
In order to prove \eqref{hbeha} let us chose $x \in \Omega$ and note that $ | x - x_0 ( x, w)| \geq \delta ( x)$. Next, let us introduce a polar
coordinates change of variables in the integral \eqref{eq:hfunc}, and note the following:
\[
  h_\alpha ( x) = \int_0^{ 2 \pi} \int_0^\infty \frac{ 1}{ \eta^{ d + \alpha}} \e^{ - | x - x_0 ( x, \sigma)| / \eta} \eta^{ d - 1} \d \eta \d \sigma
\]
where $\eta$ denotes the radial variable.
Now, introducing the change of variables $r = \delta ( x) \eta$ we obtain
\begin{align}
 h_\alpha ( x)  & \leq \int_0^{ 2 \pi} \int_0^\infty \frac{ 1}{ r^{ d + \alpha}} \e^{ - \delta ( x) / r} r^{ d - 1} \d r \d \sigma \nonumber \\ 
           & \leq \int_0^{ 2 \pi} \int_0^\infty \frac{ 1}{ ( \delta ( x) \eta )^{ 1 + \alpha}} \e^{ - 1 / \eta} \delta ( x) \d \eta \d \sigma \nonumber \\
           & = \frac{ 1}{ \delta^\alpha ( x)} \int_0^{ 2 \pi} \int_0^\infty \frac{ 1}{  \eta^{ 1 + \alpha}} \e^{ - 1 / \eta} \d \eta \d \sigma, \nonumber 
\end{align}
from which \eqref{hbeha} follows. In addition, we obtain that $h_\alpha ( x)$ is finite for every $x \in \Omega$.
\end{proof}
 
In \cite{fernandez2014boundary} it has been shown that the fractional heat equation
\begin{align}
\partial_t u ( x, t) &= - c_{ d, \alpha} \, \mbox{P.V.} \int_{ \R^d} \frac{ u( x, t) - u( y, t)}{ | x - y|^{ d + \alpha}} \d y & \mbox{ in } \Omega, t > 0, \nonumber \\
u ( x, t) &= 0            & \mbox{ in } \R^d \setminus \Omega, \nonumber \\ 
u ( x, 0) &= u^{ in} ( x) & \mbox{ in } \Omega, \nonumber 
\end{align}
has a unique solution such that for any fixed $t_0 > 0$ the following estimate holds
\[
\sup_{ t \geq t_0} \bigg\| \frac{ u ( \cdot, t)}{ \delta^{ \alpha / 2} ( \cdot)} \bigg\|_{ C^\alpha ( \bar{ \Omega})} \leq C ( t_0) \| u^{ in} \|_{ L^2 ( \Omega)}.
\]
Therefore, for any fixed time $t > 0$, $u ( x, t)$ behaves like $\delta^{ \alpha / 2} ( x)$ when $x \to \partial \Omega$.

In this work we neither prove the uniqueness of weak solutions nor any H\"older regularity results, however, 
formally using
 $\varphi(x,t) = \rho(x,t){\bf 1}_{[0,T]}(t)$ in \eqref{mainfrac}  yields
\[
  \frac{1}{2}\| \rho(\cdot,T) \|^2_{ L^2 ( \Omega)} + \int_0^T \int_\Omega  h_\alpha \rho^2 dx \,dt 
  + \Gamma(\alpha+1)\int_0^T \int_{x,y:\,[x,y]\subset\Omega} \frac{(\varphi(x)-\varphi(y))^2}{|x-y|^{d+\alpha}}
   dx\,dy\,dt = \frac{1}{2}\| \rho^{in} \|^2_{ L^2 ( \Omega)}.
\]
This implies uniqueness at least formally. Also the boundedness of the second integral together with
Proposition \ref{prop:hbound} induces results on the behaviour of $\rho$ close to the boundary.
In particular for $\alpha > 1$, as a consequence of Proposition \ref{prop:hbound}, $h_\alpha$ is not integrable,
implying some decay of $\rho(x,t)$ as $\delta(x)\to 0$.
 

\bibliographystyle{siam}

\bibliography{bibliography.bib}

\begin{thebibliography}{10}

\bibitem{AcevesSanchez+2016}
{\sc P.~{Aceves-Sanchez} and L.~{Cesbron}}, {\em {Fractional diffusion limit
  for a fractional Vlasov-Fokker-Planck equation}}, ArXiv e-prints,  (2016).

\bibitem{AcevesSancheza+2016}
{\sc P.~{Aceves-Sanchez} and A.~{Mellet}}, {\em {Asymptotic analysis of a
  Vlasov-Boltzmann equation with anomalous scaling}}, ArXiv e-prints,  (2016).

\bibitem{AcevesSanchezc+2016}
{\sc P.~Aceves-Sanchez and C.~Schmeiser}, {\em Fractional-diffusion-advection
  limit of a kinetic model}.
\newblock To appear in SIAM Journal of Mathematical Analysis (2016).

\bibitem{ben2011anomalous}
{\sc N.~Ben~Abdallah, A.~Mellet, and M.~Puel}, {\em Anomalous diffusion limit
  for kinetic equations with degenerate collision frequency}, Mathematical
  Models and Methods in Applied Sciences, 21 (2011), pp.~2249--2262.

\bibitem{Benson+2001}
{\sc D.~A. Benson, R.~Schumer, M.~M. Meerschaert, and S.~W. Wheatcraft}, {\em
  Fractional dispersion, l{\'e}vy motion, and the made tracer tests}, Transport
  in Porous Media, 42 (2001), pp.~211--240.

\bibitem{Bogdan+2007}
{\sc K.~Bogdan and T.~Jakubowski}, {\em Estimates of heat kernel of fractional
  laplacian perturbed by gradient operators}, Communications in Mathematical
  Physics, 271 (2007), pp.~179--198.

\bibitem{Ludovic2016}
{\sc L.~Cesbron}, {\em Anomalous diffusion limit of kinetic equations on
  spatially bounded domains}.
\newblock Preprint.

\bibitem{Cesbron+2012}
{\sc L.~Cesbron, A.~Mellet, and K.~Trivisa}, {\em Anomalous transport of
  particles in plasma physics}, Applied Mathematics Letters, 25 (2012),
  pp.~2344--2348.

\bibitem{Chen+1998}
{\sc Z.-Q. Chen and R.~Song}, {\em Estimates on green functions and poisson
  kernels for symmetric stable processes}, Mathematische Annalen, 312 (1998),
  pp.~465--501.

\bibitem{MR1295030}
{\sc R.~Dautray and J.-L. Lions}, {\em Mathematical analysis and numerical
  methods for science and technology. {V}ol. 6}, Springer-Verlag, Berlin, 1993.
\newblock Evolution problems. II, With the collaboration of Claude Bardos,
  Michel Cessenat, Alain Kavenoky, Patrick Lascaux, Bertrand Mercier, Olivier
  Pironneau, Bruno Scheurer and R{\'e}mi Sentis, Translated from the French by
  Alan Craig.

\bibitem{MR1803225}
{\sc P.~Degond, T.~Goudon, and F.~Poupaud}, {\em Diffusion limit for
  nonhomogeneous and non-micro-reversible processes}, Indiana Univ. Math. J.,
  49 (2000), pp.~1175--1198.

\bibitem{Del+2005}
{\sc D.~del Castillo-Negrete, B.~Carreras, and V.~Lynch}, {\em Nondiffusive
  transport in plasma turbulence: a fractional diffusion approach}, Physical
  review letters, 94 (2005), p.~065003.

\bibitem{fernandez2014boundary}
{\sc X.~Fern{\'a}ndez-Real and X.~Ros-Oton}, {\em Boundary regularity for the
  fractional heat equation}, Revista de la Real Academia de Ciencias Exactas,
  Fisicas y Naturales. Serie A. Matematicas,  (2014), pp.~1--16.

\bibitem{Fukushima+2010}
{\sc M.~Fukushima, Y.~Oshima, and M.~Takeda}, {\em Dirichlet forms and
  symmetric Markov processes}, vol.~19, Walter de Gruyter, 2010.

\bibitem{Habetler+1975}
{\sc G.~J. Habetler and B.~J. Matkowsky}, {\em {Uniform asymptotic expansions
  in transport theory with small mean free paths, and the diffusion
  approximation}}, Journal of Mathematical Physics, 16 (1975), p.~846.

\bibitem{Jara+2009}
{\sc M.~Jara, T.~Komorowski, and S.~Olla}, {\em Limit theorems for additive
  functionals of a markov chain}, The Annals of Applied Probability, 19 (2009),
  pp.~pp. 2270--2300.

\bibitem{Kwa2015}
{\sc M.~{Kwa{\'s}nicki}}, {\em {Ten equivalent definitions of the fractional
  Laplace operator}}, ArXiv e-prints,  (2015).

\bibitem{Larsen+1974}
{\sc E.~Larsen and J.~Keller}, {\em {Asymptotic solution of neutron transport
  processes for small free paths}}, J. Math. Phys., 15 (1974), pp.~53--157.

\bibitem{Mellet2010}
{\sc A.~Mellet}, {\em Fractional diffusion limit for collisional kinetic
  equations: a moments method}, Indiana Univ. Math. J., 59 (2010),
  pp.~1333--1360.

\bibitem{MR2763032}
{\sc A.~Mellet, S.~Mischler, and C.~Mouhot}, {\em Fractional diffusion limit
  for collisional kinetic equations}, Arch. Ration. Mech. Anal., 199 (2011),
  pp.~493--525.

\bibitem{DiNezza2012521}
{\sc E.~D. Nezza, G.~Palatucci, and E.~Valdinoci}, {\em Hitchhikerʼs guide to
  the fractional sobolev spaces}, Bulletin des Sciences Math{\'e}matiques, 136
  (2012), pp.~521 -- 573.

\bibitem{Servadei+2014}
{\sc R.~Servadei and E.~Valdinoci}, {\em On the spectrum of two different
  fractional operators}, Proceedings of the Royal Society of Edinburgh: Section
  A Mathematics, 144 (2014), pp.~831--855.

\bibitem{Song+2003}
{\sc R.~Song and Z.~Vondra{\v{c}}ek}, {\em Potential theory of subordinate
  killed brownian motion in a domain}, Probability Theory and Related Fields,
  125 (2003), pp.~578--592.

\bibitem{Vazquez2014857}
{\sc J.-L. V\'azquez}, {\em Recent progress in the theory of nonlinear
  diffusion with fractional laplacian operators}, Discrete and Continuous
  Dynamical Systems - Series S, 7 (2014), pp.~857--885.

\bibitem{Wigner1961}
{\sc E.~Wigner}, {\em Nuclear Reactor Theory}, AMS, 1961.

\end{thebibliography}

\end{document}